\documentclass[a4paper]{article}
\usepackage{amsthm,amsmath,amssymb}
\usepackage{enumerate,enumitem}
\usepackage{graphicx}
\usepackage{xfrac}
\usepackage{ulem}
\usepackage{mathtools}
\usepackage{ wasysym }
\usepackage{tabularx,booktabs,multirow}
\usepackage{hyperref}
\newcolumntype{C}{>{\centering\arraybackslash}X}
\newcolumntype{D}{>{\centering\arraybackslash}X}

\DeclareGraphicsExtensions{.pdf,.png,.eps}
\graphicspath{{figs/}}

\newtheorem{theorem}{Theorem}
\newtheorem{lemma}[theorem]{Lemma}

\newtheorem{proposition}[theorem]{Proposition}

\newtheorem*{claim*}{Claim}

\theoremstyle{definition}

\newtheorem*{example}{Example}
\newtheorem{definition}[theorem]{Definition}
\newtheorem{construction}[theorem]{Construction}

\newcommand{\BB}{\ensuremath{\mathcal{B}}}

\newcommand{\FF}{\ensuremath{\mathcal{F}}}
\newcommand{\cF}{\ensuremath{\mathcal{F}}}

\newcommand{\cS}{\ensuremath{\mathcal{S}}}

\newcommand{\N}{\ensuremath{\mathbb{N}}}
\newcommand{\cN}{\ensuremath{\mathcal{N}}}

\newcommand{\cU}{\ensuremath{\mathcal{U}}}
\newcommand{\cV}{\ensuremath{\mathcal{V}}}

\newcommand{\cX}{\ensuremath{\mathcal{X}}}
\newcommand{\cY}{\ensuremath{\mathcal{Y}}}
\newcommand{\cZ}{\ensuremath{\mathcal{Z}}}

\newcommand{\QQ}{\ensuremath{\mathcal{Q}}}

\newcommand{\cLa}{\mathrel{\rotatebox[origin=c]{180}{$\cV$}}}

\usepackage[usenames,dvipsnames]{xcolor}

\begin{document}

\title{Poset Ramsey number $R(P,Q_n)$. II. N-shaped poset}

\author{Maria Axenovich \footnote{Karlsruhe Institute of Technology, Karlsruhe, Germany.}
\and Christian Winter \footnote{Karlsruhe Institute of Technology, Karlsruhe, Germany. E-mail: \textit{christian.winter@kit.edu}} }
%

\maketitle

\begin{abstract}
Given partially ordered sets (posets) $(P, \leq_P)$ and $(P', \leq_{P'})$, we say that $P'$ contains a copy of $P$ if for some injective function $f\colon P\rightarrow P'$ and for any $A, B\in P$, $A\leq _P B$ if and only if $f(A)\leq_{P'} f(B)$.
For any posets $P$ and $Q$, the poset Ramsey number $R(P,Q)$ is the least positive integer $N$ such that no matter how the elements of an $N$-dimensional Boolean lattice are colored in blue and red, there is either a copy of $P$ with all blue elements or a copy of $Q$ with all red elements. 

We focus on the poset Ramsey number $R(P, Q_n)$ for a fixed poset $P$ and an $n$-dimensional Boolean lattice $Q_n$, as $n$ grows large.
It is known that $n+c_1(P) \le R(P,Q_n) \le c_2(P) n$, for positive constants $c_1$ and $c_2$. However, there is no poset $P$ known, for which $R(P, Q_n)> (1+\epsilon)n$, for $\epsilon >0$. 
This paper is devoted to a new method for finding upper bounds on $R(P, Q_n)$ using a duality between copies of $Q_n$ and sets of elements that cover them, referred to as blockers. We prove several properties of blockers and their direct relation to the Ramsey numbers. Using these properties we show that $R(\cN,Q_n)=n+\Theta(n/\log n)$, for 
a poset $\cN$ with four elements $A, B, C, $ and $D$, such that $A<C$, $B<D$, $B<C$, and the remaining pairs of elements are incomparable.
\end{abstract}

\section{Introduction}

A partially ordered set, shortly a \textit{poset}, is a set $P$ equipped with a relation $\le_P$ that is transitive, reflexive, and antisymmetric. 
For any non-empty set $\cZ$, let $\QQ(\cZ)$ be the \textit{Boolean lattice} of \textit{dimension} $|\cZ|$ on a \textit{ground set} $\cZ$, i.e.\ 
the poset consisting of all subsets of $\cZ$ equipped with the inclusion relation $\subseteq$. 
We use $Q_n$ to denote a Boolean lattice with an arbitrary $n$-element ground set. 
We refer to a poset either as a pair $(P, \leq_P)$, or, when it is clear from context, simply as a set $P$. When clear from context, we shall write $A\leq B$ instead of $A \leq_P B$, and 
$A<B$ when $A\leq B$ and $A\neq B$. When $A$ and $B$ are not comparable, we write $A\parallel B$. The elements of $P$ are often called \textit{vertices}.\\

A poset $P_1$ is an \textit{(induced) subposet} of $P_2$
if $P_1\subseteq P_2$ and for every $X_1,X_2\in P_1$, $X_1 \leq_{P_1} X_2$ if and only if $X_1 \leq_{P_2} X_2.$
An \textit{(induced) copy} of a poset $P_1$ in $P_2$ is an induced subposet $P'$ of $P_2$, isomorphic to $P_1$. We shall be considering three special posets: 
$\cV$, $\cLa$, and $\cN$, see Figure \ref{fig:VLaN}. The poset $\cV$ has three vertices $A, B,$ and $C$, $C<A$, $C<B$, and $A\parallel B$. The poset $\cLa$ has three vertices $A, B,$ and $C$, $C>A$, $C>B$, and $A\parallel B$. The poset $\cN$ has four vertices $A, B, C, D$ and relations $A<C$, $B<D$, $B<C$, $A\parallel B$, $A\parallel D$, and $C \parallel D$. \\

\begin{figure}[h]
\centering
\includegraphics[scale=0.6]{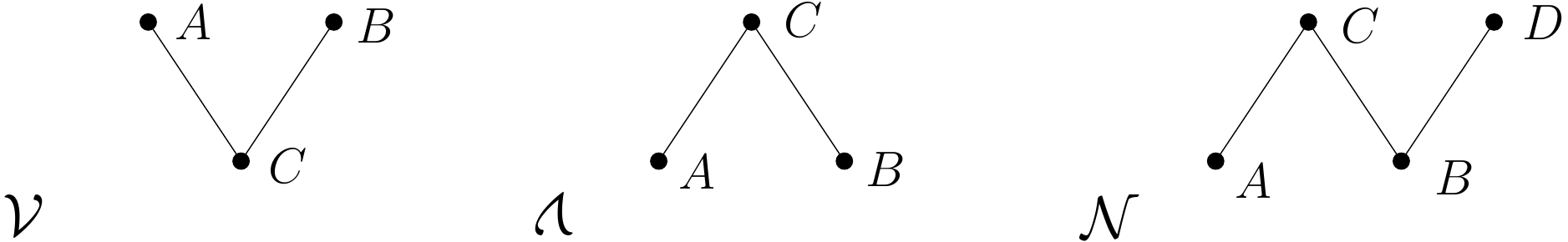}
\caption{ Hasse diagrams of posets $\cV$, $\cLa$, and $\cN$ }
\label{fig:VLaN}
\end{figure}

Extremal properties of posets and their induced subposets have been investigated in recent years and mirror similar concepts in graphs.
Carroll and Katona \cite{CK} initiated the consideration of so called Tur\'an-type problems for induced subposets. 
Most notable is a result by Methuku and P\'alvölgyi \cite{MP} which provides an asymptotically tight bound on the maximum size of a subposet of a Boolean lattice that does not have a copy of a fixed poset $P$, for general $P$. 
Their statement has been refined for several special cases, see e.g.\ Lu and Milans \cite{LM}, and M\'eroueh \cite{M}.
Further Tur\'an-type results are, for example, given by Methuku and Tompkins \cite{MT}, and Tomon \cite{T}. 
Note that Tur\'an-type properties are also investigated in depth for non-induced, so called \textit{weak} subposets, which are not considered here.
Besides that, saturation-type extremal problems are studied for induced and weak subposets, see a recent survey of Keszegh, et al.\ \cite{KLMPP}.
\\

In this paper we are dealing with Ramsey-type properties of induced subposets in Boolean lattices. 
Consider an assignment of two colors, \textit{blue} and \textit{red}, to the vertices of posets. Such a coloring $c: P \rightarrow \{blue, red\}$ is a \textit{blue/red coloring} of $P$.
A colored poset is \textit{monochromatic} if all of its vertices share the same color. A monochromatic poset whose vertices are blue is called a \textit {blue poset}. Similarly defined is a \textit{red poset}.
Extending the classical definition of graph Ramsey numbers, Axenovich and Walzer \cite{AW} introduced the \textit{poset Ramsey number} which is defined as follows. For posets $P$ and $Q$, let 

\begin{multline*}
R(P,Q)=\min\{N\in\N \colon \text{ every blue/red coloring of $Q_N$ contains either}\\ 
\text{a blue copy of $P$ or a red copy of $Q$}\}.
\end{multline*}

One of the central questions in this area is to determine $R(Q_n, Q_n)$. The best bounds currently known are 
$2n+1 \leq R(Q_n, Q_n) \leq n^2 -n+2$, see listed chronologically Walzer \cite{W}, Axenovich and Walzer \cite{AW}, Cox and Stolee \cite{CS},  Lu and Thompson \cite{LT}, Bohman and Peng \cite{BP}. 
It should be highlighted that the upper bound on $R(Q_n, Q_n)$ shows that $R(P,Q)$ is well-defined for any $P$ and $Q$ because any poset is contained as a copy in a Boolean lattice $Q_n$ for sufficiently large $n$.
\\

One subject of research on poset Ramsey numbers is the off-diagonal setting $R(P,Q_n)$ for a fixed poset $P$ and large $n$.
As general bounds the first author and Walzer \cite{AW} showed the following.
The \textit{height} $h(P)$ of a poset $P$ is defined as the size of the longest chain in $P$.
The \textit{$2$-dimension} $\dim_2(P)$ of a poset $P$ is the dimension of the smallest Boolean lattice containing a copy of $P$.
It is an easy observation that the $2$-dimension is well-defined for any $P$.

\begin{proposition}[Axenovich-Walzer \cite{AW}]
Let $P$ be a fixed poset. Then
$$n+h(P)-1 \le R(P,Q_n) \le h(P)n+\dim_2(P).$$
\end{proposition}

Here, the lower bound is trivial and is obtained by a coloring of $Q_{n+h(P)-1}$ in which all vertices in each layer $\{X\in Q_{n+h(P)-1}: |X|=\ell\}$, $\ell\le n+h(P)-1$, have the same color, red or blue, and there are $n$ red layers and $h(P)-1$ blue layers.
\\

For the off-diagonal setting $R(Q_m, Q_n)$ with $m$ fixed and $n$ large, an exact result is only known if $m=1$. It is easy to see that $R(Q_1, Q_n)=n+1$.
For $m=2$, it was shown in \cite{AW} that $R(Q_2, Q_n) \leq 2n+2$. This was improved by Lu and Thompson \cite{LT} to $R(Q_2, Q_n)\leq (5/3)n +2$, and by Gr\'osz, Methuku, and Tompkins \cite{GMT} who showed for $\epsilon>0$ and $n\in\N$ being sufficiently large:
$n+3 \leq R(Q_2,Q_n) \le n + \frac{(2+\epsilon)n}{\log n}.$ Finally, the present authors \cite{QnV} proved a lower bound asymptotically matching the upper one:

\begin{theorem}[Grosz-Methuku-Tompkins \cite{GMT}, Axenovich-Winter \cite{QnV}]\label{Q2}
$$R(Q_2,Q_n)=n+\Theta\big(\tfrac{n}{\log n}\big).$$
\end{theorem}
Further known bounds on poset Ramsey numbers include results of Chen et al. \cite{CCCLL}, \cite{CCLL}, Chang et al. \cite{CGLMNPV} as well as Winter \cite{QnA}.
\\

It is unknown whether there exists a poset $P$ such that $R(P,Q_n)\ge (1+c)n$ for some $c>0$. 
Therefore it is natural to consider the value of $R(P,Q_n)-n$ and determine its asymptotic behaviour. 
We say that a \textit{tight bound} on $R(P,Q_n)$ is a function $f(n)$ such that $R(P,Q_n)=n+\Theta(f(n))$.
A tight bound is only known for a handful of posets, see for example Theorem \ref{Q2}, and Winter \cite{QnK}.
\\

A poset is \textit{trivial} if it does not contain a copy of either $\cV$ or $\cLa$. Otherwise we refer to it as \textit{non-trivial}.
For trivial posets $P$ the trivial lower bound is asymptotically tight.

\begin{theorem}[Axenovich-Winter  \cite{QnV}]\label{thm:QnV}
If $P$ is a trivial poset, then $R(P,Q_n)=n+c(P)$ where $c(P)$ is a constant only depending on $P$.
If $P$ is a non-trivial poset, then $R(P,Q_n)\ge R(\cV,Q_n)\ge n+\frac{n}{15\log n}$.
\end{theorem}

For non-trivial posets $P$, there are only two known approaches to find the upper bound of a tight bound on $R(P,Q_n)$.
The first was introduced by Grosz, Methuku and Tompkins \cite{GMT} for an upper bound on $R(Q_2,Q_n)$ and is based on the following idea.
In a blue/red coloring of the host lattice, there is either a red copy of $Q_n$ (and we are done) or there are many blue chains.
Then using these chains counting arguments can be applied to force a particular monotonically blue structure.

An alternative approach is given in \cite{QnV} by the present authors for proving an upper bound on $R(\cV,Q_n)$.
With a careful analysis of the blue subposet of a hosting lattice with forbidden red $Q_n$ one can obtain much more information than the existence of many chains. 
In this paper we will elaborate on the second approach and formulate the central, intermediate step as a theorem for general $P$. This approach involves so-called {\it blockers}, posets that contain a vertex from each copy of $Q_n$ from a special, easier to analyse subclass. We show in Theorem~\ref{blockers-Ramsey} that extremal properties of $P$-free blockers immediately give an upper bound on $R(P, Q_n)$. Our result on $\cN$-poset then follows:

\begin{theorem}\label{thm_N}\label{thm:N}
$$n+\frac{n}{15\log n}\le R(\cN,Q_n)\le n+\frac{(1+o(1))n}{\log n}.$$
\end{theorem}

Here, the lower bound follows immediately from Theorem \ref{thm:QnV}, so the focus of this paper is on the upper bound. 

The paper is structured as follows. 
In Section \ref{sec:definitions} main definitions and basic results are given. 
Section \ref{sec:blockers} deals with the main tool used - blockers. 
In Section \ref{sec:N} a proof of Theorem~\ref{thm:N} is given.
\\

\section{Main definitions and tools}\label{sec:definitions}

A vertex $Z$ of a poset $\cF$ is a {\it minimum} of $\cF$ if it is the unique minimal element of $\cF$, i.e.\ $Z\le F$ for every $F\in\cF$. 
Similarly, a \textit{maximum} of $\cF$ is a unique maximal vertex of $\cF$.
Given a fixed poset $P$, a poset $\cF$ is \textit{$P$-free} if it contains no (induced) copy of $P$.
Let $\cX$ and $\cY$ be disjoint sets. For a subposet $\cF\subseteq\QQ(\cY)$, the \textit{$\cX$-shift} of $\cF$ is the poset $\cF'$ with vertices $\{Y\cup\cX : Y\in\cF\}$ ordered by inclusion. Note that $\cF'$ is isomorphic to $\cF$.
\\

Let $\cF_1$ and $\cF_2$ be two disjoint posets. The \textit{parallel composition} of $\cF_1$ and $\cF_2$ is the poset on vertices $\cF_1\cup\cF_2$ such that 
pairs of vertices in $\cF_1$, as well as pairs of vertices in $\cF_2$ are comparable if and only if they are likewise comparable in $\cF_1$ or $\cF_2$, respectively, 
and any two $F_1\in\cF_1$ and $F_2\in\cF_2$ are incomparable. 
In the literature this poset is also referred to as the \textit{independent union} of $\cF_1$ and $\cF_2$.
If for a poset $\cF$ there exists a partition $\cF=\cF_1\cup\cF_2$ into non-empty subposets $\cF_1$ and $\cF_2$ such that $\cF$ is the parallel composition of $\cF_1$ and $\cF_2$, we say that $\cF$ is \textit{disconnected}. Otherwise we say that $\cF$ is \textit{connected}.
\\

A \textit{weak homomorphism} of a poset $\cF_1$ into another poset $\cF_2$ is a function $\phi\colon \cF_1\to \cF_2$ 
such that for any two $A,B\in\cF_1$ with $A\le_{\cF_1} B$, we have $\phi(A)\le_{\cF_2}\phi(B)$. 
Similarly, a function $\phi\colon \cF_1\to \cF_2$ is a \textit{strong homomorphism} if for any $A,B\in\cF_1$, $A\le_{\cF_1} B$ if and only if $\phi(A)\le_{\cF_2}\phi(B)$. 
An injective weak [strong] homomorphism is a \textit{weak [strong] embedding} of $\cF_1$ into $\cF_2$.
Here we exclusively consider \textit{strong} embeddings and \textit{weak} homomorphisms, so we usually simply refer to them as ``embeddings'' and ``homomorphisms'', respectively.
\\

Throughout this paper, we consider a set $\cZ$ as the ground set of our hosting lattice $\QQ(\cZ)$ where $|\cZ|=N$ for some integer $N$.
We then partition $\cZ$ into two disjoint sets $\cX$ and $\cY$, $|\cY|\neq\varnothing$, such that $|\cX|=n$ and $|\cY|=k$ for some integers $n$ and $k$, i.e.\ $N=n+k$.
A (strong) embedding $\psi\colon \QQ(\cX)\to \QQ(\cZ)$ is \textit{$\cX$-good} if for every $X\subseteq\cX$, $\psi(X)\cap \cX=X$.
We say that a copy $Q$ of $Q_n$ in $\QQ(\cZ)$ is \textit{$\cX$-good} if there exists an $\cX$-good embedding of $\QQ(\cX)$ with image $Q$. See Figure \ref{fig:blocker_lattice} (a) for a $\{1,2\}$-good copy of $Q_2$ in $\QQ(\{1,2,x_1,x_2\})$.
Moreover, we say that $\cX$ is a {\it defining set} for a copy of $Q_n$ if this copy is $\cX$-good. One of the main structural observations we have is the following:
\begin{lemma}[Axenovich-Walzer \cite{AW}]\label{lem:embed}
Let $n\in\N$. 
Any copy of $Q_n$ in $\QQ(\cZ)$ is $\cX$-good for some subset $\cX\subseteq\cZ$ with $|\cX|=n$.
\end{lemma}

We shall also need some definitions to describe $\cN$-free posets. Let $\cF_1$ and $\cF_2$ be two disjoint posets. 
The \textit{series composition} of $\cF_1$ and $\cF_2$, $\cF_1$ \textit{below} $\cF_2$, is the poset on vertices $\cF_1\cup\cF_2$, where 
pairs of vertices in $\cF_1$, as well as pairs of vertices in $\cF_2$ are comparable if and only if they are likewise comparable in $\cF_1$ or $\cF_2$, respectively, 
and $F_1< F_2$ for any $F_1\in\cF_1$ and $F_2\in\cF_2$. We usually refer to this poset as the \textit{series decomposition of $\cF_1$ below $\cF_2$}. 
A poset is \textit{series-parallel} if either it is a $1$-element poset or it is obtained by series composition or parallel composition of two series-parallel posets. 
Valdes \cite{Valdes} showed the following characterization.

\begin{theorem}[Valdes \cite{Valdes}]\label{prop:Nfree}
A non-empty poset is $\cN$-free if and only if it is series-parallel.
\end{theorem}
%
%

\section{$\cY$-blockers}\label{sec:blockers}

\subsection{Definition and examples of $\cY$-blockers}

{\bf Outline of the main idea.} The definition of $R(P, Q_n)$ implies that there is a coloring of $Q(\cZ)$, $|\cZ| \leq R(P, Q_n)-1$, in blue and red such that the blue vertices ``cover" all copies of $Q_n$, i.e.\ there is a blue vertex in each copy of $Q_n$ and there is no copy of $P$ having only blue vertices, i.e.\ the set of blue vertices is $P$-free.
We shall classify all copies of $Q_n$ according to their defining sets and consider the set of only those blue vertices that ``cover" copies of $Q_n$ with a specific fixed defining set $\cX$. We refer to the poset induced by blue vertices as a \textit{$\cY$-blocker}, where $\cY=\cZ\setminus \cX$. We shall derive several properties of $\cY$-blockers in general and those that are $P$-free and will bound $R(P, Q_n)$ in terms of blockers. This generalises an approach used in \cite{QnV}, where ${\cLa}$-free $\cY$-blockers were considered and called \textit{$\cY$-shrubs}.

\begin{definition}
Let $\cY$ and $\cZ$ be two non-empty sets such that $\cY\subseteq \cZ$. A {\it $\cY$-blocker} in $\QQ(\cZ)$ is a subposet $\FF$ in $\QQ(\cZ)$ 
which contains a vertex from every $\cX$-good copy of $\QQ(\cX)$, where $\cX= \cZ\setminus \cY$.
We say that a $\cY$-blocker $\FF$ in $\QQ(\cZ)$ is \textit{critical} if for any vertex $F\in\cF$ the subposet $\FF\setminus\{F\}$ is not a $\cY$-blocker in $\QQ(\cZ)$.
\end{definition}

Note that for any $\cY\subseteq\cZ$, a $\cY$-blocker in $\QQ(\cZ)$ exists, for example take $\cF=\QQ(\cZ)$. 
Later on we consider ``thinner'' $\cY$-blockers satisfying special properties, in particular being $P$-free. 

\begin{example}
 Let  $\cZ=\{1,2,x_1,x_2\}$, $\cY=\{1,2\}$ and $\cX=\{x_1,x_2\}$. 
 Let $\cF$ be the $\{x_1\}$-shift of $\QQ(\cY)$, see Figure \ref{fig:blocker_lattice} (a).
Consider an arbitrary $\cX$-good copy $Q$ of $\QQ(\cX)$ in $\QQ(\cZ)$, with a corresponding $\cX$-good embedding $\psi\colon \QQ(\cX)\to \QQ(\cZ)$. 
Then $\psi(\{x_1\})=\{x_1\}\cup Y$ for some $Y\subseteq\cY$, and hence $\psi(\{x_1\})\in\cF$. Thus $\cF$ is a $\cY$-blocker in $\QQ(\cX\cup\cY)$.
Figure \ref{fig:blocker_lattice} (b) also depicts a $\cY$-blocker, we shall verify this using Theorem \ref{thm:blocker}.
\end{example}

\begin{figure}[h]
\centering
\includegraphics[scale=0.6]{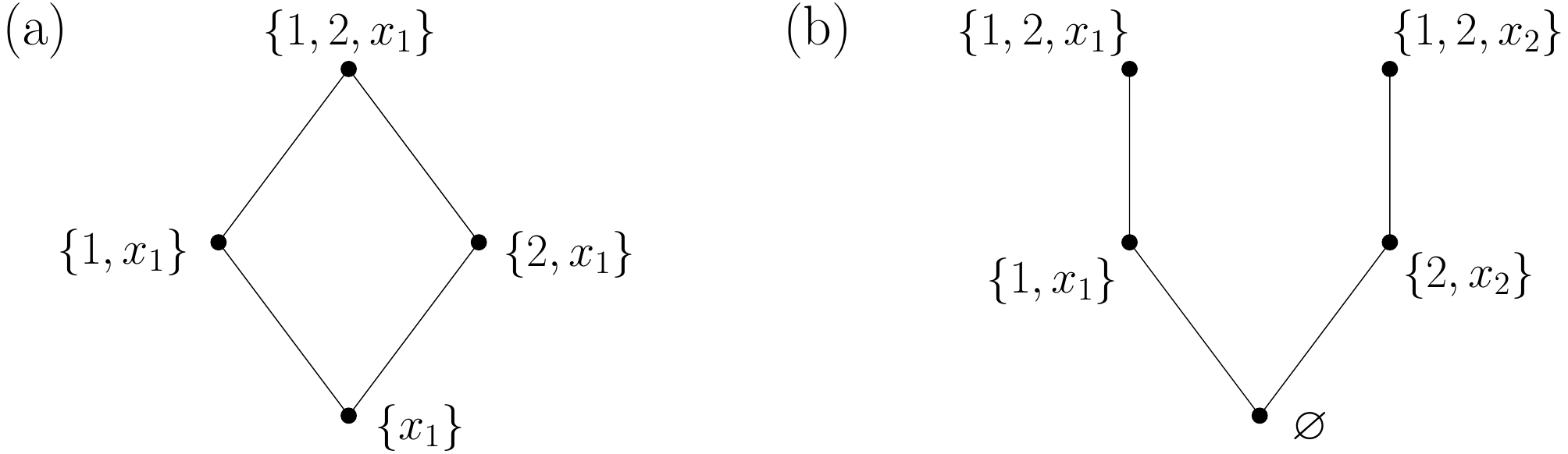}
\caption{ Two $\{1,2\}$-blockers in $\QQ(\{1,2,x_1,x_2\})$ }
\label{fig:blocker_lattice}
\end{figure}
%


\subsection{General properties of $\cY$-blockers}

\begin{lemma}\label{basic-properties}
\begin{itemize}
\item[(i)] A blue/red colored Boolean lattice $\QQ(\cZ)$ contains no red copy of $Q_n$ if and only if for each $\cX\subseteq \cZ$ of size $|\cX|=n$, there is a $\cZ\setminus \cX$-blocker with all vertices blue.
\item[(ii)] Let $\cF$ be a $\cY$-blocker where $\cY \neq \varnothing$ and let $Y\subseteq\cY$. Then there is a vertex $Z\in\cF$ with $Z\cap\cY=Y$. 
In particular, if $Z$ is a minimum of $\cF$, then $Z\cap\cY=\varnothing$; and if $Z$ is a maximum of $\cF$, then $Z\cap\cY=\cY$.
\item[(iii)] If $\cF$ is a $\cY$-blocker, then $|\cF|\geq 2^{|\cY|}$.
\end{itemize}
\end{lemma}


\begin{proof}
Part $(i)$ follows immediately from Lemma \ref{lem:embed} and the definition of a $\cY$-blocker.
For $(ii)$, let $\cF$ be a $\cY$-blocker in $\QQ(\cZ)$ and $\cX=\cZ\setminus \cY$. Observe that $\cF$ contains a vertex $U$ with $U\cap\cY=Y$ for every $Y\subseteq \cY$,
because otherwise the $\{Y\}$-shift of $\QQ(\cX)$ is an $\cX$-good copy of $\QQ(\cX)$ that does not contain a vertex from $\cF$.
Considering $Y=\varnothing$, we have that there is $U\in \cF$ such that $U\cap \cY= \varnothing$. 
Then a minimum $Z$ of $\cF$ has $\cY$-part $Z\cap \cY \subseteq U\cap \cY= \varnothing$. Similarly, a maximum $Z$ of $\cF$ has $\cY$-part $Z\cap\cY=\cY$.
For $(iii)$, since there are $2^{|\cY|}$ subsets of $\cY$, part $(ii)$ immediately implies that $|\cF|\ge 2^{|\cY|}$.
\end{proof} 


\begin{theorem}\label{thm:mPk}\label{blockers-Ramsey}
Let $P$ be a poset and let $n\in\N$ be an integer.
Then $$R(P,Q_n)\le \min \{N: ~\text{there is no }  \text{P-free} ~  \text{\cY\!-blocker in }\QQ([N])\text{ for some } \cY\subseteq [N], |\cY|=N-n\}.$$
\end{theorem}

\begin{proof}
Let $N$ be the smallest integer such that for some $\cY\subseteq[N]$, $|\cY|=N-n$, there is no $P$-free $\cY$-blocker in $\QQ([N])$.

Consider an arbitrarily blue/red colored Boolean lattice $\QQ([N])$ and let $\BB$ be the induced subposet of $\QQ([N])$ consisting of all blue vertices.
We shall show that there is either a blue copy of $P$ or a red copy of $Q_n$ in the coloring of $\QQ([N])$.
Let $\cX=[N]\setminus\cY$. 
If in $\QQ([N])$ there is a monochromatic red copy of $Q_n$ which is $\cX$-good, the proof is complete.
Otherwise each $\cX$-good copy of $Q_n$ contains a blue vertex, i.e.\ the blue subposet $\BB$ is a $\cY$-blocker.
By the definition of $N$, $\BB$ is not $P$-free. Thus there is a blue copy of $P$ in $\QQ([N])$.
\\

It remains to show that this minimum is well-defined, i.e.\ we shall find an integer $N$ such that there is no $P$-free $\cY$-blocker in $\QQ([N])$, 
where $\cY\subseteq [N]$ with $|\cY|=N-n$. 
In order to show this, we bound the size $|\cF|$ of a $P$-free $\cY$-blocker $\cF$ in $\QQ([N])$  from above and from below.
On the one hand, by a result of Methuku and P\'alv\"olgyi \cite{MP} we find that the size of the $P$-free subposet $\cF\subseteq\QQ([N])$ is bounded by
 $$|\cF|\le c(P)\binom{N}{N/2}\le \frac{c'(P)\cdot 2^{N}}{ \sqrt{N/2}},$$
 where $c$ and $c'$ are constants depending only on $P$.
On the other hand, Lemma \ref{basic-properties} provides that $$|\cF|\ge 2^{|\cY|}=2^{N-n}.$$
For sufficiently large $N$, we have that $\frac{\sqrt{N/2}}{c'(P)}> 2^{n}$, which implies that there is no $P$-free $\cY$-blocker $|\cF|$ in $\QQ([N])$.
\end{proof}

\begin{definition}
For a subposet $\FF$ of $\QQ(\cZ)$ and $\cY\subseteq \cZ$, we say that a (weak) homomorphism $\phi\colon \FF\to \QQ(\cY)$ is \textit{$\cY$-hitting} if there exists some $F\in\cF$ with $\phi(F)= F\cap\cY$.
Conversely, $\phi$ is \textit{$\cY$-avoiding} if $\phi(F)\neq F\cap\cY$ for every $F\in\FF$.
\end{definition}

\noindent\textbf{Remark. } 
In the following we show an equivalence between the existence of a $\cY$-blocker and the existence of  a $\cY$-avoiding homomorphism.
One can think of the homomorphism $\phi\colon \FF \to\QQ(\cY)$ as a ``recipe'' encoding an embedding function $\psi$ which corresponds to an $\cX$-good copy of $Q_n$ in $\QQ(\cX\cup\cY)$.
Recall that a $\cY$-blocker is defined as a poset which has a vertex in common with every $\cX$-good copy of $Q_n$, therefore every ``recipe'' $\phi$ forces a ``collision'' with $\FF$ in a vertex $F\in\FF$ with $\phi(F)=F\cap \cY$. 
However, there is no $1$-to-$1$ correspondence between functions $\phi$ and $\psi$, 
and the presented constructions building $\psi$ from $\phi$ as well as $\phi$ from $\psi$ are not inverse of each other.

\begin{theorem}\label{thm:blocker}
Let $ \cY$ be a non-empty subset of a set $\cZ$.  A subposet $\FF$ of a Boolean lattice $\QQ(\cZ)$ is a $\cY$-blocker if and only if every (weak) homomorphism $\phi\colon \FF \to\QQ(\cY)$ is $\cY$-hitting. 
\end{theorem}

\begin{example} Let  $\cZ=\{1,2,x_1,x_2\}$ and $\cY=\{1,2\}$. 
In the Boolean lattice $\QQ(\{1,2,x_1,x_2\})$ consider the subposet $\cF$ on vertices $\varnothing, \{1,x_1\}, \{1,2,x_1\}, \{2,x_2\}, \{1,2,x_2\}$, see Figure \ref{fig:blocker_lattice} (b). 
Then $\cF$ is a $\{1,2\}$-blocker. 
In order to prove this, we can use Theorem \ref{thm:blocker}. Assume towards a contradiction that there is a homomorphism $\phi\colon \FF \to\QQ(\cY)$ such that for every $F\in\FF$ we have $\phi(F)\neq F\cap \cY$.
Then $\phi(\varnothing)\cap\cY\neq \varnothing$, say without loss of generality $1\in \phi(\varnothing)\cap\cY$.
Since $\phi$ is a homomorphism $\phi(\varnothing)\subseteq \phi(\{1,x_1\})$, so $1\in \phi(\{1,x_1\})\cap\cY$.
Now, because $\phi(\{1,x_1\})\cap\cY\neq \{1\}$, we obtain that $\phi(\{1,x_1\})\cap\cY= \{1,2\}$. 
Then using that $\phi(\{1,x_1\})\subseteq \phi(\{1,2,x_1\})$, we obtain $\phi(\{1,2,x_1\})\cap\cY= \{1,2\}$, a contradiction.
\end{example}

\begin{proof}[Proof of Theorem \ref{thm:blocker}]
Let $\cY\subseteq\cZ$ be a non-empty subset and let $\cX=\cZ\setminus\cY$.
For the first part of the proof let $\cF$ be a subposet in $\QQ(\cZ)$ such that every homomorphism $\phi\colon \FF\to \QQ(\cY)$ is $\cY$-hitting.
We shall show that $\cF$ is a $\cY$-blocker.
Let $Q$ be an arbitrary $\cX$-good copy of $\QQ(\cX)$ in $\QQ(\cZ)$ with a corresponding $\cX$-good embedding $\psi\colon \QQ(\cX)\to \QQ(\cZ)$.
Consider the function $\phi\colon \FF\to \QQ(\cY)$ given by $\phi(F):=\psi(F\cap\cX)\cap\cY$ for each $F\in\cF$.
Using the properties of $\psi$, it is easy to see that if $F\subseteq F'$ for $F,F'\in \FF$, then $\phi(F)\subseteq \phi(F')$, so $\phi$ is a homomorphism.
Thus $\phi$ is $\cY$-hitting and we find some $Z\in\cF$ with $\phi(Z)=Z\cap\cY$.
Then $\psi(Z\cap\cX)\cap\cY=\phi(Z)=Z\cap\cY$. Since $\psi$ is $\cX$-good, we know that $\psi(Z\cap\cX)\cap\cX=Z\cap\cX$. Therefore $\psi(Z\cap\cX)=Z$.
Since the image of $\psi$ is $Q$, we obtain $Z=\psi(Z\cap\cX)\in Q$, thus $\cF$ and $Q$ have the vertex $Z$ in common.\\

From now on let $\cF$ be a subposet in $\QQ(\cZ)$ for which there exists a $\cY$-avoiding homomorphism $\phi\colon \FF\to \QQ(\cY)$. 
We shall show that $\cF$ is not a $\cY$-blocker. For that we shall construct an $\cX$-good embedding $\psi\colon \QQ(\cX)\to \QQ(\cZ)$ such that the image of $\psi$ does not contain a vertex from $\cF$.
Fix some $X\in\QQ(\cX)$, now we define $\psi(X)$ using an iteration: 
Informally spoken, in step $i$ we introduce a set $f_i(X)\subseteq\cY$ and check whether $X\cup f_i(X)$ is a ``feasible'' choice for $\psi(X)$; and if not, we extend $f_i(X)$ to its strict superset $f_{i+1}(X)$ and repeat.

Let $f_0(X)=\varnothing$. For $i\in\N$, let $\cF_i(X)=\{Z\in\cF : Z\subseteq X\cup f_{i-1}(X)\}$ be the down-set of $X\cup f_{i-1}(X)$ and let $f_i(X)=\bigcup _{Z\in\cF_i(X)} \phi(Z)$.
Clearly $f_i(X)\subseteq \cY$. Note that $\varnothing=f_0(X)\subseteq f_1(X)$, thus $\cF_1(X)\subseteq \cF_2(X)$ and so $f_{1}(X)\subseteq f_2(X)$.
Iteratively, we obtain that $\cF_{i}(X)\subseteq \cF_{i+1}(X)$ and $f_{i}(X)\subseteq f_{i+1}(X)\subseteq\cY$, see Figure \ref{fig:Yhitting} (a).
Thus after finitely many steps $f_j(X)=f_{j+1}(X)$ for some $j\in\N$, i.e.\ this set is ``feasible'', and let $j(X)$ be the minimal such index $j$. 
Observe that $f_{j(X)}(X)=f_{j(X)+1}(X)=f_{j(X)+2}(X)=\dots$ as $\cF_{j(X)+1}(X)=\cF_{j(X)+2}(X)=\dots$. We set $\psi(X):=X\cup f_{j(X)}(X).$\\

\begin{figure}[h]
\centering
\includegraphics[scale=0.6]{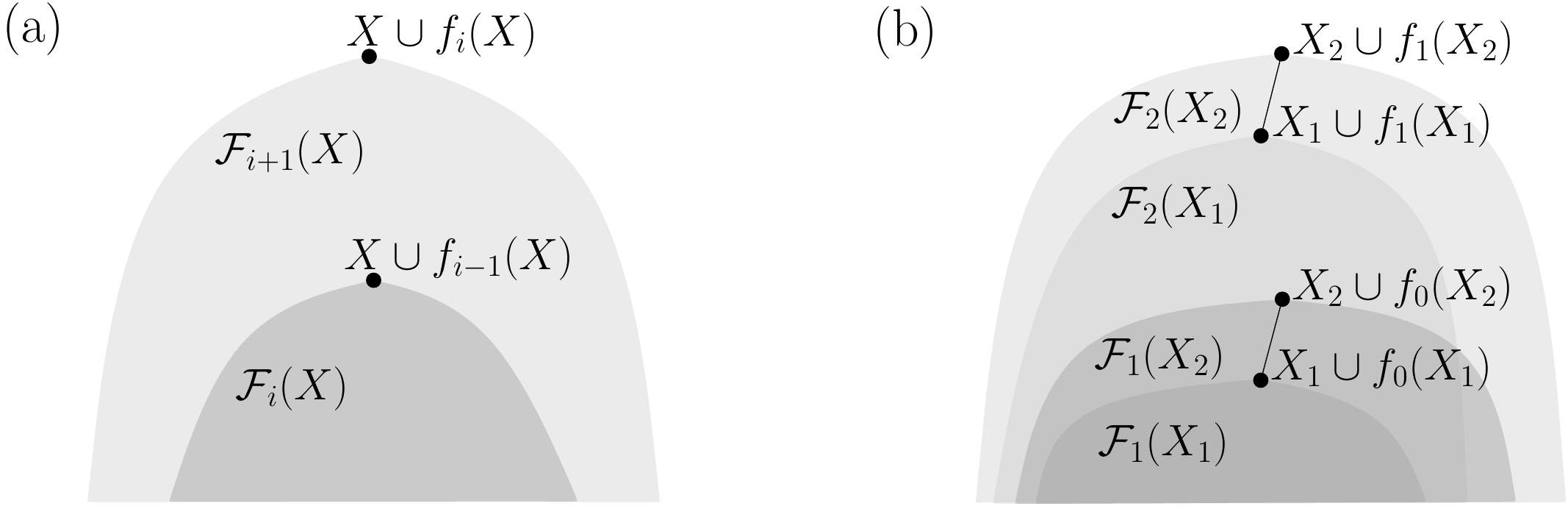}
\caption{(a) Construction of $\psi(X)$, (b) Iteration in Claim 1.}
\label{fig:Yhitting}
\end{figure}

\noindent\textbf{Claim 1.} $\psi$ is an $\cX$-good embedding of $\QQ(\cX)$.\\
\textit{Proof of Claim 1.} 
Note that for every $X\in\QQ(\cX)$, we have $f_{j(X)}(X)\subseteq \cY$ and so $\psi(X)\cap\cX=X$. 
Thus it remains to show that $\psi$ is an embedding in order to prove the claim.
Let $X_1,X_2\in\QQ(\cX)$. We shall show that $X_1\subseteq X_2$ if and only if $\psi(X_1)\subseteq\psi(X_2)$.

First suppose that $X_1\subseteq X_2$. Then $X_1\cup f_0(X)=X_1\subseteq X_2=X_2\cup f_0(X)$, so $\cF_1(X_1)\subseteq \cF_1(X_2)$. 
This implies that $f_1(X_1)\subseteq f_1(X_2)$, so $X_1\cup f_1(X)\subseteq X_2\cup f_1(X)$, see Figure \ref{fig:Yhitting} (b).
Iteratively, $\cF_i(X_1)\subseteq \cF_i(X_2)$ and $f_i(X_1)\subseteq f_i(X_2)$.
We obtain that $$f_{j(X_1)}(X_1)=f_{\max\{j(X_1),j(X_2)\}}(X_1)\subseteq f_{\max\{j(X_1),j(X_2)\}}(X_2)=f_{j(X_2)}(X_2),$$ thus $\psi(X_1)\subseteq\psi(X_2)$.
Now suppose that $\psi(X_1)\subseteq\psi(X_2)$. Then in particular $X_1=\psi(X_1)\cap\cX\subseteq \psi(X_2)\cap\cX=X_2$, so $X_1\subseteq X_2$.
\\

\noindent\textbf{Claim 2.} The image of $\psi$ contains no vertex from $\cF$.\\
\textit{Proof of Claim 2.} 
Let $X\in\cX$ and assume that $\psi(X)\in\cF$. 
We shall find a contradiction by considering $\phi(\psi(X))$.
Observe that $\psi(X)=X\cup f_{j(X)}(X)\in \cF_{j(X)+1}(X)$ and so $\phi(\psi(X))\subseteq f_{j(X)+1}(X)$.
Since $\phi$ is $\cY$-avoiding, $\phi(\psi(X))\neq \psi(X)\cap \cY=f_{j(X)}(X)=f_{j(X)+1}(X)$. 
Consequently, there exists an element $a\in f_{j(X)+1}(X)\setminus \phi(\psi(X))$.
By definition of $f_i(X)$, we find a vertex $Z\in \cF_{j(X)}(X)\subseteq \cF$ with $a\in \phi(Z)$.
Now, since $Z\in \cF_{j(X)}(X)$, we obtain that $Z\subseteq X\cup f_{j(X)}(X)=\psi(X)$ whereas element $a$ witnesses $\phi(Z)\not\subseteq \phi(\psi(X))$.
This contradicts the fact that $\phi$ is a homomorphism.
So, indeed, $\cF$ is not a $\cY$-blocker.
\\

This concludes the proof of Theorem \ref{thm:blocker}.
\end{proof}

In the following we use the characterization from Theorem \ref{thm:blocker} to analyse properties of critical blockers. 
Recall that for $\cY\subseteq\cZ$, a $\cY$-blocker $\FF$ in $\QQ(\cZ)$ is \textit{critical} if for any vertex $F\in\cF$ the subposet $\FF\setminus\{F\}$ is not a $\cY$-blocker in $\QQ(\cZ)$.

\begin{lemma}\label{lem:minimal-connected} 
Let $\cF$ be a critical $\cY$-blocker for a non-empty set $\cY$. Then $\cF$ is a connected poset. 
\end{lemma}

\begin{proof}
Assume that $\cF$ is the parallel composition of two non-empty posets $\cF_1$ and $\cF_2$, i.e.\ $\cF_1$ and $\cF_2$ are vertex-wise incomparable in $\cF$.
Then each of $\cF_1$ and $\cF_2$ is not a $\cY$-blocker by criticality of $\cF$.
Thus there are $\cY$-avoiding homomorphisms $\phi_1\colon \cF_1\to\QQ(\cY)$ and $\phi_2\colon \cF_2\to\QQ(\cY)$. 
Now the function $\psi:\cF\to\QQ(\cY)$, $$\psi(F)=\begin{cases}\phi_1(F)\text{, if }F\in\cF_1\\ \phi_2(F)\text{, if }F\in\cF_2.\end{cases}$$
is a homomorphism of $\cF$ and $\cY$-avoiding. Recall that $\cF$ is a $\cY$-blocker, so this is a contradiction to Theorem \ref{thm:blocker}.
\end{proof}

\begin{lemma} \label{lem:chain-no-Y} Let $\cF$ be a critical $\cY$-blocker for a non-empty set $\cY$. Let $U_1,U_2\in\cF$ with $U_1\neq U_2$.
If either $U_1\cap\cY=\varnothing=U_2\cap\cY$ or $U_1\cap\cY=\cY=U_2\cap\cY$, then $U_1$ and $U_2$ are not comparable.
\end{lemma}

\begin{proof}
Assume that $U_1\cap\cY=\varnothing=U_2\cap\cY$ and $U_1\subseteq U_2$. 
As $\cF$ is a critical $\cY$-blocker, the poset $\cF'=\cF\setminus\{U_2\}$ is not a $\cY$-blocker, so by Theorem \ref{thm:blocker} we find a $\cY$-avoiding homomorphism $\phi:\cF'\to \QQ(\cY)$. Let $\cU=\{U\in\cF'\colon U\subseteq U_2,\ U\neq U_2\}$, note that $\cU\neq\varnothing$, see Figure \ref{fig:root} (a).
We extend $\phi$ to a function $\psi\colon \cF\to \QQ(\cY)$ by defining 
$$\psi(F)=\begin{cases}\phi(F),& \text{ if }F\neq U_2\\ \bigcup_{U\in\cU}\phi(U),\quad &\text{ if }F= U_2.\end{cases}$$
In order to reach a contradiction, it remains to show that $\psi$ is a $\cY$-avoiding homomorphism.
We shall show that $\psi$ is a homomorphism by considering any two $F_1, F_2\in \cF$ such that $F_1\subseteq F_2$ and verifying that $\psi(F_1)\subseteq \psi(F_2)$. We need to consider cases whether either of $F_1$ or $F_2$ is equal to $U_2$. We repeatedly use the fact that $\phi$ is a homomorphism:

$\circ$ If $F_1\neq U_2$ and $F_2\neq U_2$, then $\psi(F_1)=\phi(F_1)\subseteq \phi(F_2)=\psi(F_2)$.

$\circ$ If  $F_1=U_2$, then $\psi(F_1)=\psi(U_2) = \bigcup_{U\in \cU} \phi(U) \subseteq \bigcup_{U\in \cU} \phi(F_2) \subseteq\phi(F_2)= \psi(F_2)$. Here we used the property that for any $U\in \cU$, $U\subseteq U_2\subseteq F_2$. 

$\circ$ If $F_2=U_2$, then $\psi(F_1) =\phi(F_1) \subseteq \bigcup_{U\in \cU} \phi(U)  = \psi(U_2)=\psi(F_2)$. 
Here, we used that $F_1 \in \cU$ and thus $F_1\subseteq \bigcup_{U\in \cU} U$.
Therefore, $\psi$ is a homomorphism.\\

To show that $\psi$ is $\cY$-avoiding, we need to show that for any $F\in \cF$, $\psi(F)\neq F\cap \cY$. 
 Consider first $F\in \cF$ with $F\neq U_2$, i.e.\ $F\in \cF'$. 
Since $\phi$ is $\cY$-avoiding,  $\phi(F)\neq F\cap\cY$. Since $\psi(F)=\phi(F)$, we have that $\psi(F) \neq F\cap\cY$. 
Now, let $F=U_2$. Since $\psi(U_2)\supseteq \phi(U_1)$ where $\phi(U_1)\neq U_1\cap \cY=\varnothing$, we find that $\psi(U_2)\neq\varnothing=U_2\cap\cY$.
We conclude that $\psi$ is $\cY$-avoiding. This contradicts Theorem \ref{thm:blocker} and the fact that $\cF$ is a $\cY$-blocker. 

Under the assumption that $U_1\cap\cY=\cY=U_2\cap\cY$ and $U_1\subseteq U_2$, a symmetric proof holds for $\cU=\{U\in\cF\setminus\{U_1\}\colon U\supseteq U_1,\ U\neq U_1\}$ and $\psi\colon \cF\to \QQ(\cY)$ with 
$$\psi(F)=\begin{cases}\phi(F),& \text{ if }F\neq U_1\\ \bigcap_{U\in\cU}\phi(U),\quad &\text{ if }F= U_1.\end{cases}$$
\end{proof}

\begin{figure}[h]
\centering
\includegraphics[scale=0.6]{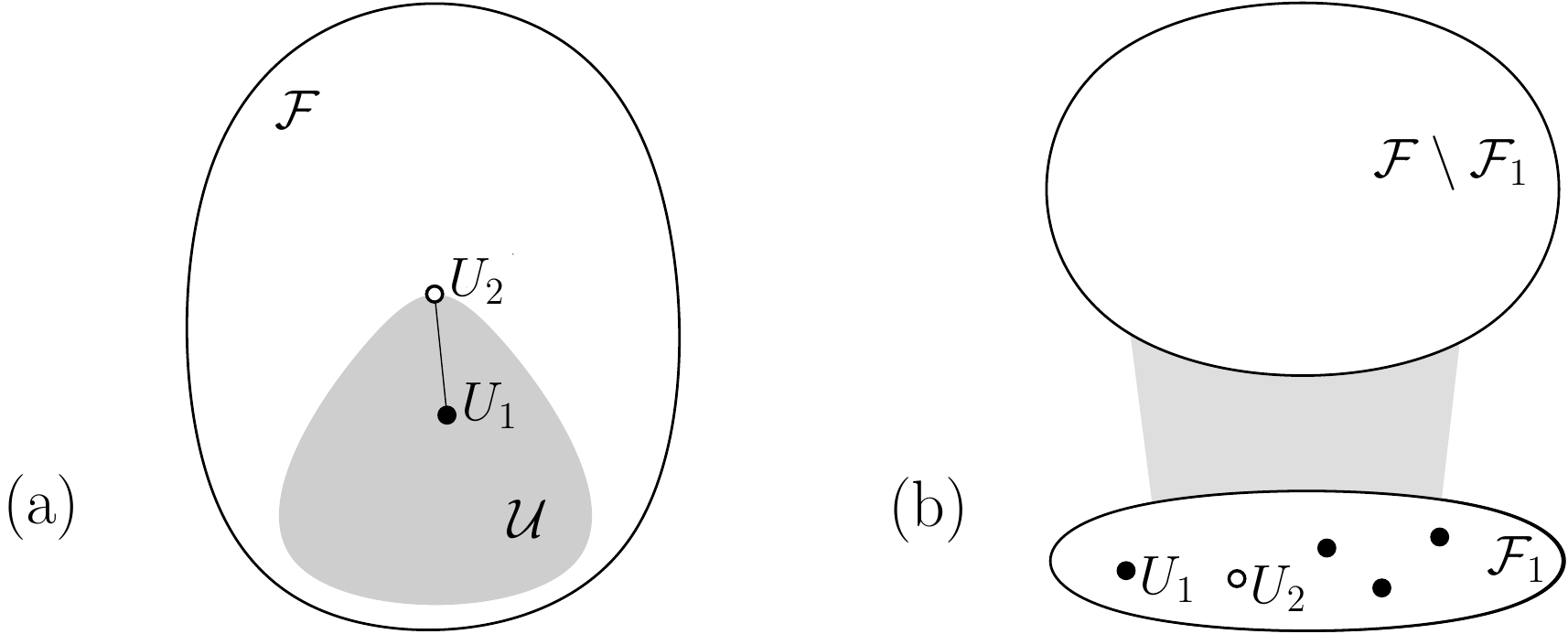}
\caption{(a) Setting in Lemma \ref{lem:chain-no-Y}, (b) Setting in Lemma \ref{lem:antichain-no-Y}.}
\label{fig:root}
\end{figure}

\begin{lemma}\label{lem:antichain-no-Y}
Let $\cF$ be a critical $\cY$-blocker where $\cY \neq \varnothing$. 
Let $\cF_1\subseteq \{U\in \cF: U\cap \cY=\varnothing\}$ such that $\cF$ is a series composition of $\cF_1$ below $\cF\setminus \cF_1$, then $|\cF_1|\le1$.
Similarly, let $\cF_2\subseteq\{U\in \cF: U\cap \cY=\cY\}$ such that $\cF$ is a series composition of $\cF\setminus \cF_2$ below $\cF_2$, then $|\cF_2|\le1$.
%
\end{lemma}

\begin{proof}
For the first part, assume towards a contradiction that there are two distinct vertices $U_1,U_2\in\cF_1$. 
 Since $\cF$ is a critical $\cY$-blocker, there is a $\cY$-avoiding homomorphism $\phi\colon \cF\setminus\{U_2\}\to\QQ(\cY)$.
Let $\psi\colon \cF\to \QQ(\cY)$ such that
$$\psi(F)=\begin{cases}\phi(F),\quad &\text{ if }F\neq U_2\\ \phi(U_1),& \text{ if }F= U_2.\end{cases}$$
We shall prove that $\psi$ is a $\cY$-avoiding homomorphism of $\cF$. 
By Lemma \ref{lem:chain-no-Y}, $\cF_1$ is an antichain. 
In order to show that $\psi$ is a homomorphism, we consider two arbitrary $F_1, F_2\in \cF$ with $F_1\subseteq F_2$ and show that $\psi(F_1)\subseteq \psi(F_2)$.

$\circ$ If $F_2=U_2$, then in particular $F_2\in\cF_1$ and so $F_1\in\cF_1$, because $\cF$ is a series composition of $\cF_1$ below $\cF\setminus \cF_1$.
Since $\cF_1$ is an antichain, we obtain that $F_1=U_2=F_2$. Then trivially $\psi(F_1)=\psi(U_2)= \psi(F_2)$.

$\circ$ If $F_1=U_2$ and $F_2\neq U_2$, we know that $F_2\in\cF\setminus\cF_1$ because $\cF_1$ is an antichain. 
Then $U_1\subseteq F_2$. Because $\phi$ is a homomorphism and by definition of $\psi$, we obtain that $\psi(U_2)=\phi(U_1)\subseteq \phi(F_2)$.

$\circ$ If $F_1\neq U_2$ and $F_2\neq U_2$, then $\psi(F_1)=\phi(F_1)\subseteq \phi(F_2)=\psi(F_2)$.
Thus $\psi$ is a homomorphism of $\cF$.\\

For every $F\in\cF\setminus\{U_2\}$, we know that $\psi(F)=\phi(F)\neq F\cap\cY$. Furthermore, $\psi(U_2)=\phi(U_1)\neq U_1\cap \cY=\varnothing=U_2\cap\cY$. 
Thus $\psi$ is $\cY$-avoiding, a contradiction.

If we assume that there are distinct $U_1,U_2\in\cF_2$, a symmetric argument considering the same function $\psi\colon \cF\to \QQ(\cY)$, 
$$\psi(F)=\begin{cases}\phi(F),\quad &\text{ if }F\neq U_2\\ \phi(U_1),& \text{ if }F= U_2.\end{cases}$$
yields a contradiction.
\end{proof}
%
%
%

\begin{lemma}\label{lem:chain_blocker}
Let $\cX$ and $\cY$ be two disjoint sets with $|\cY|=1$. Let $\cF$ be a critical $\cY$-blocker in $\QQ(\cX\cup\cY)$.
Then $\cF$ is a chain consisting of two vertices $X_1, X_2\cup\cY$, where $X_1\subseteq X_2\subseteq\cX$.
\end{lemma}

\begin{proof}Since $|\cY|=1$, we find that for every $Z\in\cF$ either $Z\cap\cY=\varnothing$ or $Z\cap\cY=\cY$.
Consider subposets  $\cF_1 = \{Z\in \cF: ~Z\cap\cY=\varnothing\}$ and $\cF_2 = \{Z\in \cF: ~Z\cap\cY=\cY\}$ partitioning $\cF$.
Lemma \ref{basic-properties} provides that $\cF_1\neq \varnothing$ and $\cF_2\neq \varnothing$. 
By Lemma \ref{lem:minimal-connected}, $\cF$ is connected, so in particular there are two vertices from $\cF_1$ and from $\cF_2$ which are comparable.
Let these vertices be $X_1\in\cF_1$ and $X_2\cup\cY \in\cF_2$, where $X_1,X_2\subseteq \cX$. Then $X_1\subseteq X_2\cup\cY$, so $X_1\subseteq X_2$.
\\

Next we need to show that $\cF=\{X_1,X_2\cup\cY\}$.
Consider the subposet $\cF'=\{X_1,X_2\cup\cY\}\subseteq\cF$. We show that $\cF'$ is a $\cY$-blocker in $\QQ(\cX\cup\cY)$,
i.e.\ by Theorem \ref{thm:blocker} we shall show that there is no $\cY$-avoiding homomorphism from $\cF'$ to $\QQ(\cY)$.
A homomorphism $\phi:\cF'\to \QQ(\cY)$ is $\cY$-avoiding only if $\phi(X_1)=\cY$ and $\phi(X_2\cup\cY)=\varnothing$, 
but such a homomorphism does not exist, since $\phi(X_1)\subseteq\phi(X_2\cup\cY)$ because of $X_1\subseteq X_2\cup\cY$.
We obtain that $\cF'$ is a $\cY$-blocker, therefore $\cF=\{X_1,X_2\cup\cY\}$ since $\cF$ is critical.
\end{proof}

\begin{lemma}\label{lem:reduction}
Let $\cY$ be a set of size at least $2$ and let $a\in\cY$. Let $\cF$ be a $\cY$-blocker. 
Then the induced subposets $\{F\in\cF\colon a\in F\}$ and $\{F\in\cF\colon a\notin F\}$ are $(\cY\setminus\{a\})$-blockers.
\end{lemma}
\begin{proof}
Let $\cF'=\{F\in\cF\colon a\in F\}$. 
Assume that $\cF'$ is not a $(\cY\setminus\{a\})$-blocker, i.e.\ by Theorem \ref{thm:blocker} there is a $(\cY\setminus\{a\})$-avoiding homomorphism $\phi\colon \cF'\to \QQ(\cY\setminus\{a\})$.
We find a $\cY$-avoiding homomorphism of $\cF$ in order to reach a contradiction.
Let $\psi\colon \cF\to \QQ(\cY)$ with $$\psi(F)=\begin{cases}\phi(F)\cup\{a\},&\text{ if }F\in\cF'\\ \{a\},&\text{ if }F\notin\cF'.\end{cases}$$
Observe that $\psi$ is a homomorphism, because $\{a\}\subseteq \phi(F)\cup\{a\}$ for all $F\in\cF'$ and $\phi$ is a homomorphism.

For every $F\in\cF\setminus\cF'$, note that $a\in\psi(F)$ but $a\notin F\cap\cY$, thus $\psi(F)\neq F\cap\cY$.
On the other hand, recall that $\phi$ is $(\cY\setminus\{a\})$-avoiding.
Hence for every $F\in\cF'$ we know that $\phi(F)\neq F\cap (\cY\setminus\{a\})$ where $a\notin \phi(F)$ and $a\notin F\cap (\cY\setminus\{a\})$. 
This implies $$\psi(F)=\phi(F)\cup\{a\}\neq F\cap (\cY\setminus\{a\})\cup\{a\}=F\cap\cY.$$
As a result, $\psi$ is a $\cY$-avoiding homomorphism of $\cF$, which is a contradiction.
\\

The second part of the lemma follows from a symmetric argument for $\cF''=\{F\in\cF\colon a\notin F\}$ using the function $\psi\colon \cF\to \QQ(\cY)$,
$$\psi(F)=\begin{cases}\phi(F),&\text{ if }F\in\cF''\\ \cY\setminus\{a\},&\text{ if }F\notin\cF''.\end{cases}$$

%
\end{proof}


\subsection{Properties of $\cN$-free $\cY$-blockers}


\begin{theorem}\label{thm:min_vertex}
Let $\cX$ and $\cY$ be disjoint sets with $\cY\neq\varnothing$.
Let $\cF$ be an $\cN$-free, critical $\cY$-blocker in $\QQ(\cX\cup\cY)$. 
Then $\cF$ has at least one of a minimum vertex or a maximum vertex. 
\end{theorem}

\begin{proof}[Proof of Theorem \ref{thm:min_vertex}] 
Since $\cY\neq\varnothing$, Lemma \ref{basic-properties} implies that $|\cF|\geq 2^1$. 
By Theorem \ref{prop:Nfree}, $\cF$ is series-parallel, so it can be partitioned into two disjoint, non-empty posets $\cF_1$ and $\cF_2$ such that $\cF$ is either the parallel composition of $\cF_1$ and $\cF_2$ or the series composition of $\cF_1$ below $\cF_2$.
The former could not happen by Lemma \ref{lem:minimal-connected}.
Thus $\cF$ can be partitioned into two disjoint, non-empty posets $\cF_1$ and $\cF_2$ such that for every $F_1\in\cF_1$ and $F_2\in\cF_2$, $F_1\subseteq F_2$.

Let $Y_1=(\bigcup_{F\in\cF_1} F)\cap\cY$ be the $\cY$-part of the union of all vertices in $\cF_1$ and let $Y_2=(\bigcap_{F\in\cF_2} F)\cap\cY$ be the $\cY$-part of the intersection of all vertices in $\cF_2$. Clearly, $Y_1\subseteq Y_2\subseteq \cY$. 
\\

First assume that $Y_1\notin \{\varnothing,\cY\}$.
Then there are $a\in Y_1$ and $b\in\cY\setminus Y_1$.
Lemma \ref{basic-properties} provides that the $\cY$-blocker $\cF$ contains a vertex $U$ with $U\cap\cY=\{b\}$.
Then $U\notin\cF_1$ since $b\in U$ while $b\notin Y_1$, but also $U\notin\cF_2$ as $a\notin U$ while $a\in Y_1\subseteq Y_2$. 
We arrive at a contradiction, hence $Y_1\in\{\varnothing,\cY\}$. Symmetrically, $Y_2\in\{\varnothing,\cY\}$.
Take an arbitrary $\cY$-part $Y\subseteq\cY$ such that $Y\notin\{\varnothing,\cY\}$. By Lemma \ref{basic-properties} there is a vertex $Z\in\cF$ with $Z\cap\cY=Y$.
Then $Z\in\cF_1$ or $Z\in\cF_2$. In the first case we obtain that $Y_1\neq\varnothing$, thus $Y_1=\cY$ and hence $Y_2=\cY$ (because $Y_1\subseteq Y_2\subseteq\cY$). 
In the second case $Y_2\neq\cY$, so $Y_1=Y_2=\varnothing$.

Thus either $Y_1=Y_2=\varnothing$ or $Y_1=Y_2=\cY$. For the rest of the proof we suppose that $Y_1=Y_2=\varnothing$. If $Y_1=Y_2=\cY$, a symmetric argument holds.
\\

Because $Y_1=(\bigcup_{F\in\cF_1} F)\cap\cY=\varnothing$, we obtain $F\cap\cY=\varnothing$ for every $F\in\cF_1$. 
By Lemma~\ref{lem:antichain-no-Y}, there is at most one vertex in $\cF_1$. 
The unique vertex $Z\in \cF_1$ is the unique minimal vertex of $\cF$ and $Z\cap\cY=Y_1=\varnothing$.
In the case that $Y_1=Y_2=\cY$, we can argue symmetrically and obtain that $Z$ is a maximum of $\cF$ and $Z\cap\cY=Y_2=\cY$.
\end{proof}

\subsection{Construction of the family $\{(\cF_S, Z_S, A_S, B_S): ~S\in \cS\}$}

In the following proof we will define posets and vertices indexed by \textit{ordered sets}.

\begin{definition}
An \textit{ordered set} $S$ is a sequence $S=(y_1,\dots,y_m)$ of distinct elements $y_i$, $i\in[m]$. Given a set $\cY$, $S$ is an \textit{ordered subset} of $\cY$ if $y_i\in\cY$ for all $i\in[m]$.
We denote the \textit{empty} ordered set by $\varnothing_o=()$.
The underlying unordered set of $S$ is denoted by $\underline{S}$ and $|S|=|\underline{S}|$ is the \textit{size} of $S$.
For an ordered set $S=(y_1, \ldots, y_m)$ and an element $y_{m+1}\notin \underline{S}$, we write $(S, y_{m+1})$ for an ordered set $(y_1, \ldots, y_m, y_{m+1})$.
We say that an ordered set $S'$ is a \textit{prefix} of $S$ if $|S|\ge |S'|$ and each of the first $|S'|$ members of $S$ coincides with the respective member of $S'$.
Note that $\varnothing_o$ is a prefix of every ordered set. For $i\in\{0,\dots,|S|\}$, we denote by $S[i]$ the unique prefix of $S$ of size $i$.
A prefix $S'$ of $S$ is \textit{strict} if $S'\neq S$. 
For a set $\cY$ and an ordered subset $S$ of $\cY$, we denote the set of all elements of $\cY$ that are not in $S$ by $\cY-S=\cY\setminus\underline{S}$.
 \end{definition}
 
In the following we analyse the structure of an $\cN$-free critical $\cY$-blocker by selecting smaller and smaller subposets which are critical $\cY'$-blockers for some $\cY'\subseteq\cY$.
Recall that Theorem \ref{thm:min_vertex} implies that any critical $\cY'$-blocker has either a minimum or a maximum vertex, we call such a vertex a {\it root} of the blocker. 
Note that the blocker could have both a minimum vertex and a maximum vertex. In this case we select on of them to be the assigned root of the blocker and ignore the second root.

\begin{construction}\label{construction} ~
Let $\cY$ be a set with $|\cY|=k$. Let $\cF$ be an $\cN$-free, critical $\cY$-blocker in $\QQ(\cZ)$, $\cY\subseteq \cZ$. Let $\cS$ be the set of all ordered subsets of $\cY$ of size at most $k-1$.
In the following we recursively construct a family $\{(\cF_S, Z_S, A_S,B_S): ~S\in \cS\}$, where $\cF_S$ is a critical $(\cY-S)$-blocker, $\cF_S\subseteq \cF$, and $Z_S$ is the root of $\cF_S$.
In addition $A_S\cup B_S=\underline{S}$, where each element of $A_S$ is included in each vertex of $\cF_S$ and each element of $B_S$ is excluded from each vertex of $\cF_S$. The sets $A_S$ and $B_S$ are used as tools to encode crucial information on the blocker $\cF_S$ and its root $Z_S$ as well as $\cF_{S'}$ and $Z_{S'}$ for prefixes $S'$ of $S$.
If the root $Z_S$ is a minimum vertex in $\cF_S$, we say that $S$ is {\it min-type}, otherwise we say that $S$ is {\it max-type}.\\

{\bf Initial step. } Let $S= {\varnothing_o}$. In this case let $\cF_S=\cF$. Let $Z_{S}$ be an arbitrarily chosen root of $\cF$, i.e.\ a minimum or maximum of $\cF$, which exists due to Theorem \ref{thm:min_vertex}. Let $A_{S}=B_S=\varnothing$.\\

{\bf General iterative step. } Consider an arbitrary non-empty ordered subset $S$ of $\cY$ with $|S|\le k-1$. Let $S'$ be the prefix of $S$ such that $(S',a)=S$ for some $a\in\cY$.
Given $(\cF_{S'}, Z_{S'},A_{S'},B_{S'})$ such that $\cF_{S'}$ is a critical $(\cY-{S'})$-blocker, $Z_{S'}$ is a root of $\cF_{S'}$, and $A_{S'}, B_{S'}$ are disjoint sets partitioning $A_{S'}\cup B_{S'}=\underline{S'}$, we shall construct $\cF_{S}$, $Z_{S}$, $A_{S}$, and $B_{S}$.
By Lemma~\ref{lem:reduction} and since $(\cY-{S'})\setminus\{a\}=\cY-S$, the sets $\{F\in \cF_{S'}: a \in F\}$ and $\{F\in\cF_{S'}: a\notin F\}$ induce $(\cY-S)$-blockers. \\

If ${S'}$ is min-type, we define $\cF_{S}$ to be an arbitrary critical $(\cY-S)$-blocker which is an induced subposet of $\{F\in \cF_{S'}: a \in F\}$. Note that $a\in F$ for every $F\in\cF_{S}$. Let $A_{S} = A_{S'}\cup \{a\}$ and $B_S=B_{S'}$. 

If ${S'}$ is max-type, we define $\cF_{S}$ to be an arbitrary critical $(\cY-S)$-blocker which is an induced subposet of $\{F\in \cF_{S'}: a \notin F\}$. Note that in this case $a\notin F$ for every $F\in\cF_{S}$. Let $A_S=A_{S'}$ and $B_{S}=B_{S'}\cup \{a\}$. \\

It remains to select $Z_{S}$. Theorem \ref{thm:min_vertex} provides the existence of a root in $\cF_{S}$.
If $|{S}|\le k-2$, let $Z_{S}$ be an arbitrary root of $\cF_{S}$. If $|S|=k-1$, we need to be more careful in choosing $Z_{S}$.  
We have that $\cF_{S}$ is a critical $(\cY-S)$-blocker, for $|\cY-S|=1$. 
By Lemma \ref{lem:chain_blocker}, $\cF_{S}$ has exactly two vertices, a minimum and a maximum.
If ${S'}$ is min-type, let $Z_{S}$ be the minimum of $\cF_{S}$, i.e.\ ${S}$ is min-type.
If ${S'}$ is max-type, let $Z_{S}$ be the maximum of $\cF_{S}$, here ${S}$ is max-type.
\\

\begin{figure}[h]
\centering
\includegraphics[scale=0.6]{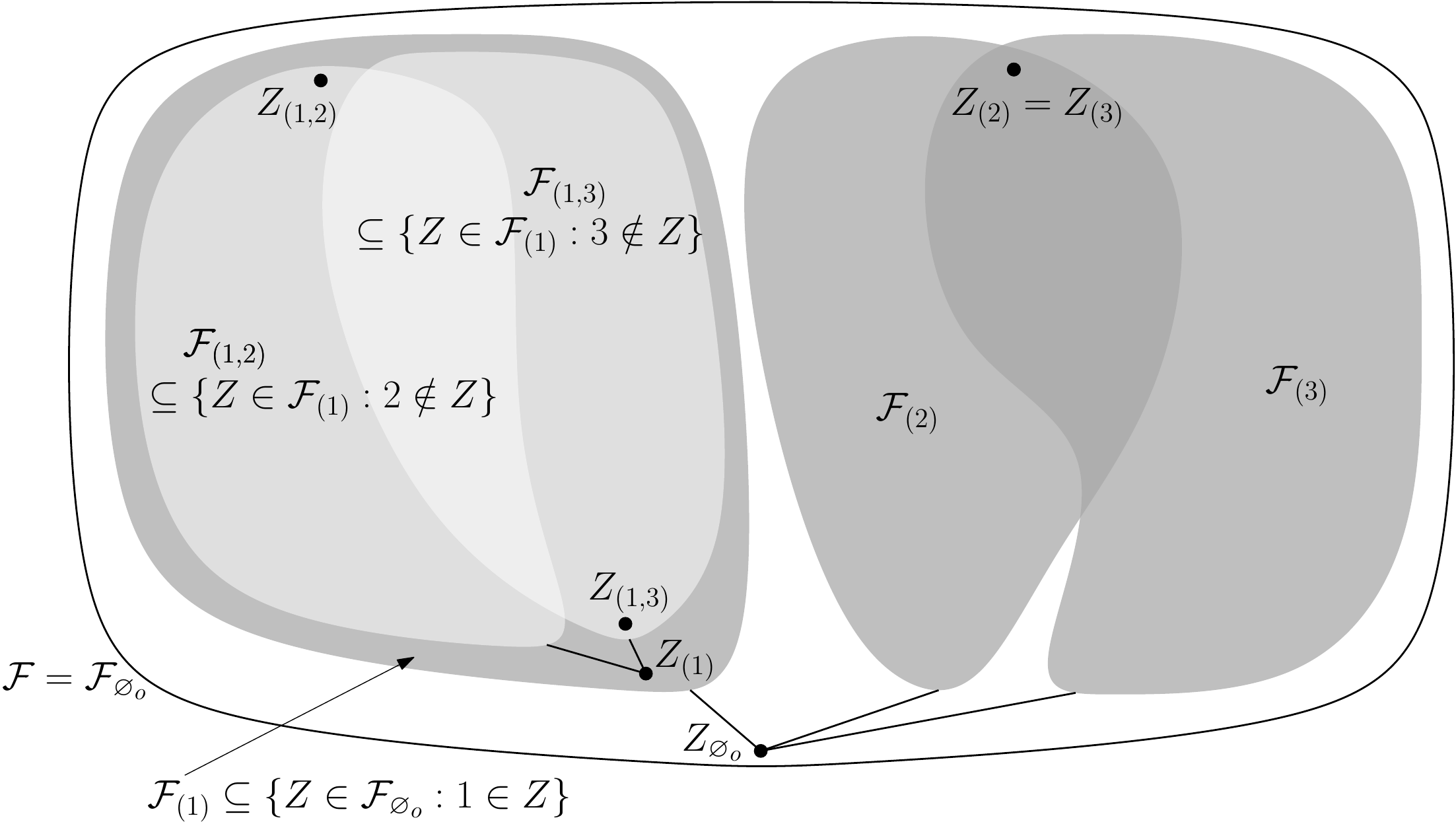}
\caption{Exemplary construction of $\cF_{(1,3)}$ and $\cF_{(1,2)}$ for $\cY=\{1,2,3\}$.}

\label{fig:recursion}
\end{figure}

The construction terminates after all ordered subsets of $\cY$ of size at most $k-1$ have been considered. 
The family $\{(\cF_S, Z_S,A_S,B_S): ~S\in \cS\}$ gives a recursive structural decomposition of $\cF$ into ``up'' and ``down'' components, i.e.\ max-type and min-type blockers, as illustrated in Figure \ref{fig:recursion}. Note that blockers $\cF_S$ may heavily overlap. Several properties follow immediate from the construction.
\end{construction}

\begin{lemma}\label{lem:construction}\label{lem:Ypart}
Let $S$ be an ordered subset of $\cY$ of size at most $k-1$ and let $S'$ be a prefix of $S$. Then
\begin{enumerate}
\item[(i)] $\cF_{S'}\subseteq\cF_S$, $A_{S'}=A_S\cap \underline{S'}$ and $B_{S'}=B_S\cap \underline{S'}$.
\item[(ii)] The size of the set $A_S$ is equal to the number of min-type strict prefixes $S'$ of $S$. The size of $B_S$ is equal to the number of max-type strict prefixes $S'$ of $S$.
\item[(iii)]  If $S$ is min-type, $\cY\cap Z_{{S}} = \{y\in\cY: y\in Z_{S}\}= A_{S}$.
If $S$ is max-type, $\cY\setminus Z_{{S}} = \{y\in\cY: y\notin Z_{S}\} = B_{S}$.
\end{enumerate}
\end{lemma}

\begin{proof}$(i)$ and $(ii)$ are easy to see. For part $(iii)$, recall that $\cF_{S}$ is a $(\cY-S)$-blocker. 
If ${S}$ is min-type, then $Z_{S}$ is a minimum of $\cF_{S}$, so $Z_{S}\cap(\cY-S)=\varnothing$ by Lemma \ref{basic-properties}.
Thus $$Z_{{S}}\cap\cY= Z_{S}\cap \underline{S}=A_{S}.$$
Similarly, if $\cF_{S}$ is max-type, then Lemma \ref{basic-properties} provides $Z_{S}\cap(\cY-S)=(\cY-S)$. 
Hence $$Z_{{S}}\cap\cY= \big(Z_{S}\cap \underline{S}\big) \cup (\cY-S)=A_{S}\cup (\cY-S)=\cY\setminus B_{S},$$
therefore $\cY\setminus Z_{{S}}=B_{S}$.
\end{proof}

\begin{lemma}
Let $S$ be an ordered subset of $\cY$ of size $k-1$ and let $S'$ be a strict prefix of $S$.
Then $Z_S\cap(\cY- S')\notin \{\varnothing,\cY-S'\}$. 
\end{lemma}


\begin{proof} Note that $|\cY-S|=1$, so let $\cY-S=\{b\}$.
First we consider the case that $|S'|=k-2$, i.e.\ $S=(S',a)$ for some $a\in\cY$. Note that $\cY-S'=\{a,b\}$. 
We shall show that one of the two elements $a,b$ is in $Z_S$ while the other is not.
\\

If $S$ is min-type, we obtain from the construction that $a\in A_S$. By Lemma \ref{lem:construction} $(iii)$, $A_S=Z_S\cap \cY$, so in particular $a\in Z_S\cap (\cY-S')$.
On the other hand, $A_S\subseteq \underline{S}$, so $b\notin A_S=Z_S\cap \cY$ and thus $b\notin Z_S\cap (\cY-S')$.

If $S$ is max-type we can argue similarly. Note that $a\in B_S$. By  Lemma \ref{lem:construction} $(iii)$, $B_S=\cY\setminus Z_S$, so $a\notin  Z_S\cap (\cY-S')$.
Furthermore, $B_S\subseteq \underline{S}$, so $b\notin B_S=\cY\setminus Z_S$. Thus $b\in Z_S\cap\cY$ and hence $b\in Z_S\cap (\cY-S')$.
\\

It remains to consider the case $|S'|<k-2$. Let $S''=S[k-2]$ be the prefix of $S$ of size $k-2$, then $S'$ is a prefix of $S''$. Observe that $\cY-S''\subseteq\cY-S'$.
We already showed that $Z_S\cap(\cY-S'')\notin \{\varnothing,\cY-S''\}$, so in particular $Z_S\cap(\cY-S')\notin \{\varnothing,\cY-S'\}$.
\end{proof}

%



\section{Proof of Theorem \ref{thm_N}}\label{sec:N}

\begin{proof}[Proof of Theorem \ref{thm_N}]
%
%
%
Let $k$ and $N$ be arbitrary integers with $N\ge k$, let $n$ such that $N=n+k$.
Let $\cY$ be a set on $|\cY|=k$ elements, say without loss of generality $\cY = \{1,\ldots, k\}$. Fix $\cZ$ with $\cY\subseteq \cZ$ and $|\cZ|=N$.
Suppose that there is an $\cN$-free, critical $\cY$-blocker $\cF$ in $\QQ(\cZ)$. 
In other words, suppose that the integer $N$ is sufficiently large with respect to $k$ such that there exists an $\cF$ with these properties in $\cZ$.

In the following we show that $\cF$ contains an antichain of size at least $k!2^{-k-1}$. Applying Sperner's theorem we obtain that $n+k=N=|\cZ|\ge k\log k$,
which implies that $k\le (1+o(1))n/ {\log n}$, i.e.\ $N=n+k\le n+(1+o(1))n/{\log n}$. Then Theorem \ref{thm:mPk} provides the required bound.
Next we argue that in $\cF$ there exists such a large antichain.
\\

Let $\cS$ be the set of all ordered subsets of $\cY$ of size at most $k-1$. 
Consider the family $\{(\cF_S, Z_S, A_S, B_S): ~S\in \cS\}$ given by Construction \ref{construction}. 
Let $\cS_1$ be the family of all ordered subsets of $\cY$ of size exactly $k-1$. 
We introduce two kinds of equivalence between elements in $\cS_1$, \textit{type-equivalence} and \textit{intersection-equivalence}.
In the following we will show the existence of a large subfamily $\cS_3\subseteq\cS_1$ such that its elements are pairwise type-equivalent but not intersection-equivalent.
We further prove that the vertices $\{Z_S: S\in\cS_3\}$ induce a large antichain in $\QQ(\cX\cup\cY)$.
\\


Let $S_1,S_2\in\cS_1$ be two ordered subsets of $\cY$ of size $k-1$.
We say that $S_1$ and $S_2$ are \textit{type-equivalent} if for any prefixes $S'_1$ of $S_1$ and $S'_2$ of $S_2$ of the same size, ${S'_1}$ is min-type if and only if ${S'_2}$ is min-type. Equivalently, ${S'_1}$ is max-type if and only if ${S'_2}$ is max-type.
The ordered sets $S_1$ and $S_2$ are \textit{intersection-equivalent} if for any same-sized prefixes $S'_1$ of $S_1$ and $S'_2$ of $S_2$, $Z_{S'_1}\cap\cY=Z_{S'_2}\cap\cY$.
It is obvious that both notions define equivalence relations on $\cS_1$. Note that intersection-equivalence of two ordered sets in $\cS_1$ is a very strong property. 
It provides a good intuition to think of intersection-equivalent ordered sets as equal. 
Several technical parts of the proof, in particular in Claim 1, arise from the fact that there might be intersection-equivalent ordered sets which are distinct.
\\

\noindent\textbf{Claim 1.} There exists a subfamily $\cS_3\subseteq\cS_1$ of size at least $2^{-k-1}k!$ such that any two distinct ordered sets $S_1,S_2\in\cS_3$, are type-equivalent but not intersection-equivalent.\smallskip\\
\textit{Proof of Claim 1.} Recall that $|\cS_1|=k!$. For every $i\in\{0,\dots,k-1\}$ and for every $S\in\cS_1$, the prefix $S[i]$ of $S$ of size $i$ is either min-type or max-type.
By pigeonhole principle for fixed $i$, there are at least $|\cS_1|/2$ ordered subsets $S\in\cS_1$ such that all prefixes $S[i]$ are of the same type.
Inductively, we find a subfamily $\cS_2\subseteq\cS_1$ of size at least $2^{-k}|\cS_1|$ such that for any fixed $i\in\{0,\dots,k-1\}$, all prefixes $S[i]$, $S\in\cS_2$ have the same type. Equivalently, the elements of $\cS_2$ are pairwise type-equivalent.
\\

In the following we show that each intersection-equivalence class in $\cS_2$ has size at most~$2$. 
Thus by selecting a representative of each equivalence class we obtain a subfamily $\cS_3$ as required.
\\

Given an arbitrary fixed $S_1\in\cS_2$, consider an ordered set $S_2\in\cS_2$ such that $S_1$ and $S_2$ are intersection-equivalent, i.e.\ $Z_{S'_1}\cap\cY=Z_{S'_2}\cap\cY$ for every two same-sized prefixes $S'_1$ of $S_1$ and $S'_2$ of $S_2$.
Without loss of generality suppose that $S_1=(1,2,\dots,k-1)$ and $\cY-S_1=\{k\}$. Let $S_2=(y_1,\dots,y_{k-1})$ and $\cY-S_2=\{y_k\}$. 
We shall show that $y_i=i$ for all but at most two indices $i\in[k]$, which implies that $S_2$ is either equal to $S_1$ or obtained from $S_1$ by interchanging the two differing members, and therefore the intersection-equivalence class of $S_1$ consists of at most $2$ members.
\\

Since $S_1$ and $S_2$ are both in $\cS_2$, i.e.\ type-equivalent, we know that for every $i\in\{0,\dots,k-1\}$ either both ${S_1[i]}$ and ${S_2[i]}$ are min-type or both ${S_1[i]}$ and ${S_2[i]}$ are max-type.
We enumerate the index set $\{0,\dots,k-1\}$ as follows.
Let $i_1,\dots,i_p$ be the indices $i\in\{0,\dots,k-1\}$ such that ${S_1[i]}$ and ${S_2[i]}$ are min-type in increasing order.
Similarly, let $j_1,\dots,j_q$ enumerate in increasing order the indices $j\in\{0,\dots,k-1\}$ where ${S_1[j]}$ and ${S_2[j]}$ are max-type.
Note that $\{i_1,\dots,i_p\}\cup\{j_1,\dots,j_q\}=\{0,\dots,k-1\}$.
\\

Now consider any two consecutive indices $i=i_{\ell}$ and $i'=i_{\ell+1}$ for some fixed $\ell\in[p-1]$. 
We know from Lemma \ref{lem:Ypart} $(iii)$ that $Z_{S_1[i]}\cap \cY=A_{S_1[i]}\text{ and }Z_{S_1[i']}\cap \cY=A_{S_1[i']}.$
Our next step is to show that $i+1$ is the unique element in the set difference of those two sets.
\\

Recall that $S_1[i]$ is min-type.
In Construction \ref{construction} in the iterative step for $S_1[i+1]=(S_1[i],i+1)$, we defined $A_{S_1[i+1]}=A_{S_1[i]}\cup\{i+1\}$.
Note that $i+1\in A_{S_1[i+1]}\subseteq A_{S_1[i']}$, while $i+1\notin A_{S_1[i]}$ because $A_{S_1[i]}\subseteq \underline{S_1[i]}=\{1,\dots,i\}$ by Lemma \ref{lem:construction} $(i)$.
Furthermore, Lemma \ref{lem:construction}~$(ii)$ provides that $|A_{S_1[i]}|=\ell$ and $|A_{S_1[i']}|=\ell+1$, thus $|A_{S_1[i']}\setminus A_{S_1[i]}|=1$ and so
$$(Z_{S_1[i']}\cap \cY)\setminus (Z_{S_1[i]}\cap \cY)=A_{S_1[i']}\setminus A_{S_1[i]}=\{i+1\}.$$
Similarly for $S_2$ we obtain
$$(Z_{S_2[i']}\cap \cY)\setminus (Z_{S_2[i]}\cap \cY)=A_{S_2[i']}\setminus A_{S_2[i]}=\{y_{i+1}\}.$$
Since $S_1$ and $S_2$ are intersection-equivalent, we find that $y_{i+1}=i+1$.
\\

We obtain that $y_{i_{\ell}+1}=i_{\ell}+1$ for every $\ell\in[p-1]$. 
For $j_1,\dots,j_q$ a symmetric argument for  $j=j_{\ell}$ and $j'=j_{\ell+1}$ considering the set difference 
$$(Z_{S_1[j]}\cap \cY)\setminus (Z_{S_1[j']}\cap \cY)=(\cY\setminus B_{S_1[j]})\setminus (\cY\setminus B_{S_1[j']})=B_{S_1[j']}\setminus B_{S_1[j]}=\{j+1\}$$
 yields that  $y_{j_{\ell+1}}=j_{\ell+1}$ for every $\ell\in[q-1]$.
Thus $y_{i+1}=i+1$ for all indices $i\in\{0,\dots,k-1\}\setminus\{i_p,j_q\}$, so $S_1$ and $S_2$ coincide in all but at most two members.
As a consequence, $S_2$ is either equal to $S_1$ or obtained from $S_1$ by interchanging the two differing members. Therefore the intersection-equivalence class of $S_1$ consists of at most $2$ ordered sets. As $S_1$ was chosen arbitrary, every intersection-equivalence class of $\cS_2$ has size at most $2$.
Select $\cS_3\subseteq\cS_2$ by choosing an arbitrary representative from each intersection-equivalence class, i.e.\ let $\cS_3$ be the largest subfamily of $\cS_2$ where every two distinct $S_2,S_3\in\cS_3$ are not intersection-equivalent. Then 
$$|\cS_3|\ge |\cS_2|/2\ge 2^{-k-1}|\cS_1|= 2^{-k-1}k!,$$
which concludes the proof of Claim 1.
\\

\noindent\textbf{Claim 2.} 
The set $\{Z_S: S\in\cS_3\}$ induces an antichain in $\QQ(\cZ)$ of size $|\cS_3|=k!2^{-k-1}$.\\ 

\noindent\textbf{Remark. } 
Although not necessary for the proof of Theorem \ref{thm_N}, Claim 2 holds in greater generality. 
Analogously to the following proof, one can obtain that for every family $\cS'$ of ordered sets such that any two distinct members of $\cS'$ are type-equivalent and not intersection-equivalent,
the vertices $\{Z_S: S\in\cS'\}$ induce an antichain in $\QQ(\cZ)$ of size $|\cS'|$.
\\

\textit{Proof of Claim 2.} Recall that any two distinct, ordered sets in $\cS_3$ are type-equivalent but not intersection-equivalent. 
We shall show that for every two distinct $S_1,S_2\in\cS_3$, the vertices $Z_{S_1}$ and $Z_{S_2}$ are incomparable.
Assume towards a contradiction that $Z_{S_1}\subseteq Z_{S_2}$.
Since $S_1$ and $S_2$ are not intersection-equivalent, there are same-sized prefixes $S'_1$ of $S_1$ and $S'_2$ of $S_2$ such that $Z_{S'_1}\cap\cY\neq Z_{S'_2}\cap\cY$.
Since $S_1$ and $S_2$ are type-equivalent, both ${S'_1}$ and ${S'_2}$ have the same type, suppose that they are min-type.
\\

First we argue that the sets $Z_{S'_1}\cap\cY$ and $Z_{S'_2}\cap\cY$ are not comparable.
Lemma \ref{lem:Ypart} $(iii)$ shows that $Z_{S'_1}\cap\cY=A_{S'_1}$ and $Z_{S'_2}\cap\cY=A_{S'_2}$.
Type-equivalence implies that pairs of same-sized prefixes of $S'_1$ and $S'_2$ always have the same type.
Thus Lemma \ref{lem:construction} $(ii)$ yields that $|A_{S'_1}|=|A_{S'_2}|$.
We obtain that $Z_{S'_1}\cap\cY=A_{S'_1}$ and $Z_{S'_2}\cap\cY=A_{S'_2}$ are distinct but of the same size, 
consequently $Z_{S'_1}\cap\cY$ and $Z_{S'_2}\cap\cY$ are not comparable.
\\

If $S'_1=S_1$ and $S'_2=S_2$, then $Z_{S_1}\cap\cY$ and $Z_{S_2}\cap\cY$ are incomparable, and so $Z_{S_1}|| Z_{S_2}$, a contradiction to our assumption that $Z_{S_1}\subseteq Z_{S_2}$.
For the remaining proof suppose that the size $|S'_1|=|S'_2|$ is strictly less than $k-1$.
We will show that there is a copy of $\cN$ in $\cF$ contradicting the definition of $\cF$ to be an $\cN$-free poset.
Let $\cY'= \cY-S'_2$, note that $\cF_{S'_2}$ is a $\cY'$-blocker.
Since $Z_{S'_1}\cap\cY$ and $Z_{S'_2}\cap\cY$ are not comparable, there exists an element $a\in Z_{S'_1}\cap\cY$ with $a\notin Z_{S'_2}$.
Lemma \ref{basic-properties} $(ii)$ yields the existence of a vertex $U\in\cF_{S'_2}$ with $U\cap\cY'=\cY'\setminus\{a\}$.
\\

\begin{figure}[h]
\centering
\includegraphics[scale=0.6]{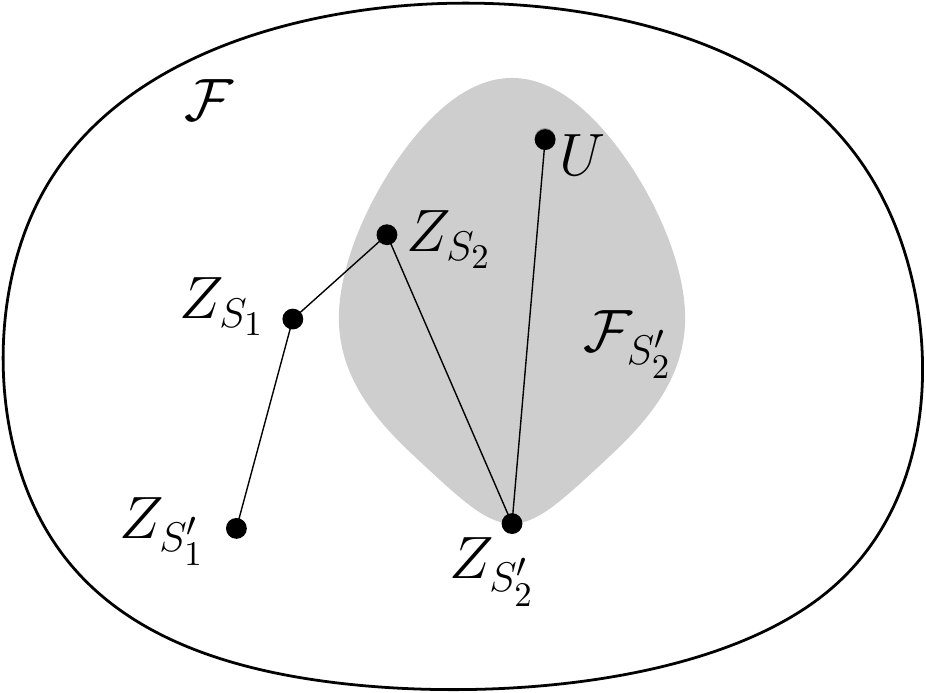}
\caption{Copy of $\cN$ constructed in the proof of Claim 2.}

\label{fig:claim5}
\end{figure}

\noindent Now we verify that $Z_{S'_1}$, $Z_{S_2}$, $Z_{S'_2}$ and $U$ form a copy of $\cN$ in $\cF$, see Figure \ref{fig:claim5}.
\begin{itemize}
\item $Z_{S'_2}\subseteq Z_{S_2}$ because $Z_{S'_2}$ is a minimum of $\cF_{S'_2}$ and $Z_{S_2}\in\cF_{S_2}\subseteq\cF_{S'_2}$ by Lemma \ref{lem:construction} $(i)$.
\item $Z_{S'_1}\subseteq Z_{S_1}\subseteq Z_{S_2}$ as $Z_{S'_1}$ is a minimum of $\cF_{S'_1}$ and $Z_{S_1}\in\cF_{S_1}\subseteq\cF_{S'_1}$ by Lemma \ref{lem:construction} $(i)$.
\item $Z_{S'_1} || Z_{S'_2}$ because $Z_{S'_1}\cap\cY$ and $Z_{S'_2}\cap\cY$ are not comparable.
\item $Z_{S'_2}\subseteq U$ because $Z_{S'_2}$ is a minimum of $\cF_{S'_2}$ and by definition $U\in\cF_{S'_2}$.
\item Note that $a\notin U$ and $a\in Z_{S'_1}$, so $Z_{S'_1}\not\subseteq U$. Since $Z_{S'_2}\not\subseteq Z_{S'_1}$ but $Z_{S'_2}\subseteq U$, transitivity yields $U\not\subseteq Z_{S'_1}$.
Therefore $U$ and $Z_{S'_1}$ are incomparable.
\item We know that $a\notin U$ but $a\in Z_{S_2}$, thus $Z_{S_2}\not\subseteq U$. Note that $Z_{S_2}\cap\cY'\neq \cY'\setminus\{a\}$ since $a\in Z_{S_2}$.
Furthermore, Lemma \ref{lem:Ypart} provides that $Z_{S_2}\cap\cY'\neq\cY'$. Thus $Z_{S_2}\cap\cY'$ is not a superset of $\cY'\setminus\{a\}=U\cap\cY'$, 
therefore $U\not\subseteq Z_{S_2}$. We obtain that $U || Z_{S_2}$.
\item The four vertices are distinct because otherwise we find an immediate contradiction to one of the above relations.
\end{itemize}
Thus, there is a copy of $\cN$ in $\cF$, which is a contradiction to the fact that $\cF$ is $\cN$-free.
\\

If $S'_1$ and $S'_2$ are max-type, a symmetric argument can be applied: As a first step, it follows similarly that $Z_{S'_1}\cap\cY$ and $Z_{S'_2}\cap\cY$ are incomparable,
and then for $\cY'=\cY-S'_1$ and for a vertex $U\in\cF_{S'_1}$ with $U\cap\cY'=\{a\}$ we find a copy of $\cN$ on vertices $Z_{S'_1}$, $Z_{S_1}$, $Z_{S'_2}$ and $U$, which is a contradiction as before. This concludes the proof of Claim 2.
\\

Claims 2 implies the existence of an antichain of size at least $k!2^{-k-1}$ in $\cF\subseteq\QQ(\cZ)$.
We define the \textit{Sperner number} $\alpha(t)$ as the smallest $m\in\N$ such that $\binom{m}{\lfloor m/2\rfloor}\ge t$. It is folklore that $\alpha(t)\ge \log(t)$. 
By Sperner's theorem \cite{S}, $|\cZ|\ge \alpha(k!2^{-k-1})$, so 
$$N=|\cZ|\ge \alpha\left(\frac{k!}{2^{k+1}}\right)\ge \log\left(\frac{k!}{2^{k+1}}\right)\ge \log\left(\frac{k^k}{2^{k+1}e^{k}}\right)\ge k(\log(k)-c),$$
for a fixed constant $c>0$. 
Recall that $N=n+k$, so we obtain that $n\ge k\log k -\big(1+c\big)k$, which implies that $k\le \big(1-o(1)\big)\frac{n}{\log n}$.
Then Theorem \ref{thm:mPk} provides that $$R(\cN,Q_n)\le N=n+k\le n+ \big(1+o(1)\big) \frac{n}{\log n}.$$ The lower bound on $R(\cN,Q_n)$ follows from Theorem~\ref{thm:QnV}.

\end{proof}

\section{Concluding remarks}
In this paper we showed that $R(\cN,Q_n)\le n+ \Theta\left(\frac{n}{\log n}\right)$.
A key ingredient in our approach is Theorem \ref{thm:mPk} where we showed a connection between the poset Ramsey number of $R(P,Q_n)$ for a poset $P$ and the extremal function $m_P(n)$ defined as
$$m_P(n):=\min \{N: ~\text{there is no }  \text{P-free} ~  \text{\cY\!-blocker in }\QQ([N])\text{ for some } \cY\subseteq [N], |\cY|=N-n\}.$$
A $\cY$-blocker can be seen as a \textit{transversal} of a set of specific Boolean lattices, and is related to other notions of \textit{transversals}, e.g.\ \textit{clique-transversals} in graphs as introduced by Erd\H{o}s, Gallai and Tuza \cite{EGT}. Seen in this context, research on $m_P(n)$ or similar extremal functions on $\cY$-blockers might be of independent interest. 
\\

In our proof of Theorem \ref{thm:N} we obtained a bound on $N=n+k$ by finding a large antichain in our $\cN$-free $\cY$-blocker $\cF$ and applying Sperner's theorem, which says that an antichain in $\QQ([N])$ has size at most $\binom{N}{\lceil N/2\rceil}$. 
Alternatively one can utilize the \textit{extremal number} of $\cN$:
Given a poset $P$ and an integer $N$, the \textit{extremal number} $\operatorname{La}^*(N,P)$ is the maximum number of vertices in a $P$-free poset in $\QQ([N])$. 
Methuku and P\'alv\"olgyi \cite{MP} showed that $\operatorname{La}^*(N,P)\le c(P)\binom{N}{N/2}$ which has the same order of magnitude as the value given by Sperner's theorem used in our proof.
Instead of showing that $\cF$ contains a large antichain, i.e.\ many pairwise incomparable vertices, one can simply show that the number of vertices in $\cF$ is large. This argument can be applied for determining $R(P,Q_n)$ for any fixed poset $P$.

\noindent \textbf{Data Availability:}~\quad  Data sharing not applicable to this article as no datasets were generated or analysed during the current study.
\\

\noindent \textbf{Acknowledgments:}~\quad  None.
\\

\noindent \textbf{Funding Statement:}~\quad  The research of both authors was partially supported by Deutsche Forschungsgemeinschaft, grant FKZ AX 93/2-1.
\\

\noindent \textbf{Conflict of Interest:}~\quad None.
\\

\noindent \textbf{Author Contribution Statement :}~\quad M.A. and C.W. wrote and reviewed the manuscript.
\\

\end{document}